\documentclass[reqno,11pt,a4paper]{article}
\usepackage{amsmath,amssymb,amsthm,color,graphicx}
\usepackage{algorithm}
\usepackage{algorithmic}

\newtheorem{theorem}{Theorem}[section]

\newtheorem{proposition}[theorem]{Proposition}
\newtheorem{definition}[theorem]{Definition}
\newtheorem{lemma}[theorem]{Lemma}

\newtheorem{remark}[theorem]{Remark}

\numberwithin{equation}{section}
\numberwithin{theorem}{section}

\renewcommand{\phi}{\varphi}

\newcommand{\eps}{\varepsilon}

\newcommand{\nada}[1]{}

\newcommand{\R}{\mathbb R}

\newcommand{\beq}{\begin{equation}}
\newcommand{\eeq}{\end{equation}}

\newcommand{\Omdg}{\Omega^{\delta,\gamma}}
\newcommand{\Omdgb}{\bar{\Omega}^{\delta,\gamma}}
\newcommand{\dg}{{\delta,\gamma}}

\definecolor{darkgreen}{rgb}{0,0.55,0}

\newcommand{\res}{\mathop{\hbox{\vrule height 7pt width .5pt depth
			0pt\vrule height .5pt width 6pt depth 0pt}}\nolimits}
\newfont{\indic}{bbmss12}

\definecolor{light}{gray}{.97}

\title{Variational approximation of functionals defined on 1-dimensional connected sets in $\R^n$}
\author{M. Bonafini\thanks{Dipartimento di  Matematica, Universit\`a di Trento, Italy, e-mail: mauro.bonafini@unitn.it},
	G. Orlandi\thanks{Dipartimento di Informatica, Universit\`a di Verona, Italy, e-mail: giandomenico.orlandi@univr.it},
	\'E. Oudet\thanks{Laboratoire Jean Kuntzmann, Universit\'e de Grenoble Alpes, France, e-mail: edouard.oudet@imag.fr}}
\date{\today}

\begin{document}
	
\maketitle

\begin{abstract}
In this paper we consider the Euclidean Steiner tree problem and, more generally, (single sink) Gilbert--Steiner problems as prototypical examples of variational problems involving 1-dimensional connected sets in $\R^n$. Following the the analysis for the planar case presented in \cite{BoOrOu}, we provide a variational approximation through Ginzburg--Landau type energies proving a $\Gamma$-convergence result for $n \geq 3$.
\end{abstract}



\section{Introduction}

Given $N$ distinct points $P_1,\dots,P_N$ in $\R^n$ and $0 \leq \alpha \leq 1$, the (single sink) Gilbert--Steiner problem, or $\alpha$-irrigation problem \cite{BeCaMo, Xia} requires to find an optimal network $L$ along which to flow unit masses located at the sources $P_1, \dots, P_{N-1}$ to the target point $P_N$, where the cost of moving a mass $m$ along a path of length $\ell$ scales like $\ell m^\alpha$. The transportation network $L$ can be viewed as $L = \cup_{i=1}^{N-1}\lambda_i$, with $\lambda_i$ a path connecting $P_i$ to $P_N$ (i.e., the trajectory of the unit mass located at $P_i$), and thus the problem translates into
$$
\inf \left\{ \int_L |\theta(x)|^\alpha d{\mathcal H}^1(x), \;\;\theta(x) = \sum_{i=1}^{N-1} \mathbf{1}_{\lambda_i}(x) \right\}
\leqno{(I_\alpha)}
$$
where $\theta$ represents the mass density along the network. In particular, $(I_0)$ reduces to the optimization of the total length of the graph $L$ and corresponds to the classical Euclidean Steiner Tree Problem (STP), i.e., finding the shortest connected graph which contains the terminal points $P_1,\dots,P_N$. For any $\alpha \in [0,1]$ a solution to $(I_\alpha)$ is known to exist and any optimal network turns out to be a tree \cite{BeCaMo}.

As pointed out in the companion paper \cite{BoOrOu}, the Gilbert--Steiner problem represents the basic example of problems defined on $1$-dimensional connected sets, and it has recently received a renewed attention in the Calculus of Variations community. In the last years available results focused on variational approximations of the problem mainly in the planar case \cite{BoLeSa,ChMeFe,OuSa,Mi-al}, while higher dimensional approximations have been recently proposed in \cite{ChMeFe2, BoBrLe}.

In this paper we extend to the higher dimensional context the two dimensional analysis developed in \cite{BoOrOu} and we propose a variational approximation for $(I_\alpha)$ in the Euclidean space $\R^n$, $n \geq 3$. We prove a result in the spirit of $\Gamma$-convergence (see Theorem \ref{thm:main} and Proposition \ref{cor:minconpsi}) by considering integral functionals of Ginzburg--Landau type \cite{ABO1, ABO2} (see also \cite{Sandier}). This approach builds upon the interpretation of $(I_\alpha)$ as a mass minimization problem in a cobordism class of integral currents with multiplicities in a suitable normed group (as studied in \cite{MaMa, MaMa2}). Thus, the relevant energy turns out to be a convex positively $1$-homogeneous functional (a norm), for which one can use calibration type arguments to prove minimality of certain given configurations \cite{MaMa2, MaOuVe}. The proposed method is quite flexible and can be adapted to a variety of situations, including manifold type ambients where a suitable formulation in vector bundles can be used (this will be treated in a forthcoming work).

Eventually, we remark that another way to approach the problem is to investigate possible convex relaxations of the limiting functional, as already pointed out in \cite{BoOrOu} and then further extended in \cite{BoOu}, so as to include more general irrigation-type problems (with multiple sources/sinks) and even problems for $1$-d structures on manifolds.

The plan of the paper is as follows. In Section \ref{sec:pre} we briefly review the main concepts needed in the subsequent sections and in Section \ref{sec:recall} we recall the variational setting for $(I_\alpha)$ relying on the concept of $\Psi$-mass. We then provide in Section \ref{sec:varapp} a variational approximation of the problem in any dimension $n \geq 3$ by means of Ginzburg--Landau type energies.

\section{Preliminaries and notations}\label{sec:pre}

In this section we fix the notation used in the rest of the paper and some basic facts. We will follow closely \cite{ABO1, ABO2}, to which we refer for a more detailed treatment.

For any $n \geq 2$, we denote by $\{e_1,\dots,e_n\}$ the standard basis of $\R^n$, $B_r^n$ is the open ball in $\R^n$ with centre the origin and radius $r$, $\mathbb S^{n-1} = \partial B_1^n$ is the unit sphere in $\R^n$, and
$$
\alpha_n = |B_1^n|, \quad  \beta_n = (n-1)^{n/2}\alpha_n,
$$
where $|\cdot|$ stands for the Lebesgue measure of the given set. For $0 \leq k \leq n$ we denote by $\mathcal{H}^k$ the $k$-dimensional Hausdorff measure. Furthermore, we assume we are given $N$ distinct points $P_1,\dots,P_N$ in $\R^n$, for $n \geq 3$ and $N \geq 2$, and we denote $A = \{P_1,\dots,P_N\}$. We also assume, without loss of generality, that $A \subset B_1^n$.

\smallskip

\textbf{Ginzburg--Landau functionals.} We consider a continuous potential $W \colon \R^{n-1} \to \R$ which vanishes only on $\mathbb S^{n-2}$ and is strictly positive elsewhere, and we require
\[
\liminf_{|y|\to 1} \frac{W(y)}{(1-|y|)^2} > 0 \quad \text{and} \quad \liminf_{|y|\to \infty} \frac{W(y)}{|y|^{n-1}} > 0.
\]
Given $\eps > 0$, $\Omega \subset \R^n$ open and $u \in W^{1,n-1}(\Omega;\R^{n-1})$, we set
\begin{equation}\label{eq:gilasingle}
F_\eps(u, \Omega) := \int_{\Omega} e_\eps(u)\,dx=\int_{\Omega} \frac{1}{n-1}|Du|^{n-1}+\frac{1}{\eps^2}W(u)\,dx,
\end{equation}
where $|Du|$ is the Euclidean norm of the matrix $Du$.

\smallskip

\textbf{Currents.}
Given $k = 0,\dots,n$, let $\bigwedge^k(\R^n)$ be the space of $k$-covectors on $\R^n$ and $\bigwedge_k(\R^n)$ the space of $k$-vectors. The canonical basis of $\bigwedge^1(\R^n)$ will be denoted as $\{dx^1,\dots,dx^n\}$. For a $k$-covector $\omega$ we define its comass as
\[
||\omega||^* = \sup \{ \omega \cdot v \,:\, v \text{ is a simple $k$-vector with } |v| = 1 \}.
\]
For $\Omega \subset \R^n$, a $k$-form on $\Omega$ is a map from $\Omega$ into the space of $k$-covectors and a $k$-dimensional current is a distribution valued into the space of $k$-vectors. We denote as $\mathcal{D}^k(\Omega)$ the space of all smooth $k$-forms with compact support and as $\mathcal{D}_k(\Omega)$ the space of all $k$-currents. In particular, the space $\mathcal{D}_k(\Omega)$ can be identified with the dual of the space $\mathcal{D}^k(\Omega)$ and equipped with the corresponding weak$^*$ topology. Furthermore, for $T \in \mathcal{D}_k(\Omega)$ and an open subset $V \subset \Omega$, we define the mass of $T$ in $V$ as
\[
||T||_V = \sup \{ T(\omega) \,:\, \omega \in \mathcal{D}^k(V), ||\omega(x)||^* \leq 1 \text{ for every } x \}
\]
and we denote the mass of $T$ as $||T|| = ||T||_\Omega$. The boundary of a $k$-current $T$ is the $(k-1)$-current characterized as $\partial T(\omega) = T(d\omega)$ for every $\omega \in \mathcal{D}^{k-1}(\Omega)$, where $d\omega$ is the exterior differential of the form $\omega$.
Let $T \in \mathcal{D}_k(\Omega)$ be a current with locally finite mass, then there exist a positive finite measure $\mu_T$ on $\R^n$ and a Borel measurable map $\tau \colon \Omega \to \bigwedge_k(\R^n)$ with $||\tau||\leq 1$ $\mu_T$-a.e., such that
\begin{equation}\label{eq:normalcurr}
T(\omega) = \int_{\R^n} \omega(x) \cdot \tau(x) \,d\mu_T(x) \quad \text{for every } \omega \in \mathcal{D}^k(\Omega).
\end{equation}
We denote $|T| = |\mu_T|$ the variation of the measure $\mu_T$, so that, given $V \subset \Omega$, one has $||T||_V := |T|(V)$. A $k$-current $T$ is said to be normal whenever both $T$ and $\partial T$ have finite mass, and we denote as $\mathbb N_k(\Omega)$ such space.

Given a $k$-rectifiable set $\Sigma$ oriented by $\tau$ and a real-valued function $\theta \in L^1_{loc}(\mathcal{H}^k\res \Sigma)$, we define the current $T = [\![\Sigma, \tau, \theta]\!]$ as
\[
T(\omega) = \int_\Sigma \theta(x) \omega(x) \cdot \tau(x) \,d\mathcal{H}^k(x),
\]
and we refer to $\theta$ as the multiplicity of the current. A $k$-current $T$ is called rectifiable if it can be represented as $T = [\![\Sigma, \tau, \theta]\!]$ for a $k$-rectifiable set $\Sigma$ and an integer valued multiplicity $\theta$. If both $T$ and $\partial T$ are rectifiable, we say $T$ is an integral current and denote as $\mathcal{I}_k(\Omega)$ the corresponding group. A polyhedral current in $\R^n$ is a finite sum of $k$-dimensional oriented simplexes $S_i$ endowed with some constant integer multiplicities $\sigma_i$, and we generally assume that $S_i \cap S_j$ is either empty of consists of a common face of $S_i$ and $S_j$.
As it is done in \cite{ABO2}, we introduce the following flat norm of a current $T \in \mathcal{D}_k(\Omega)$:
\begin{equation}
\mathbf{F}_\Omega(T) := \inf \{ ||S||_\Omega \,:\, S \in \mathcal{D}_{k+1}(\Omega) \text{ and } T = \partial S \},
\end{equation}
and the infimum is taken to be $+\infty$ if $T$ is not a boundary.

\smallskip

\textbf{Jacobians of Sobolev maps and boundaries.} Given $\Omega \subset \R^n$ open and $u\in W^{1,n-2}_{\text{loc}}(\Omega;\R^{n-1}) \cap L^\infty_{\text{loc}}(\Omega;\R^{n-1})$, following \cite{JeSo02}, we define the $(n-2)$-form
\[
j(u)=\sum_{i=1}^{n-1} (-1)^{i-1}u_{i} \cdot \bigwedge_{j\neq i} du_j
\]
and we set the Jacobian of $u$ to be
\[
Ju := \frac{1}{n-1}d[j(u)]
\]
in the sense of distributions. This means that for any $\omega \in \mathcal{D}^{n-1}(\Omega)$
\[
Ju \cdot \omega = \frac{1}{n-1} \int_{\R^n} d^*\omega \cdot j(u) \,dx,
\]
where $d^*$ is the formal adjoint of $d$. By means of the $\star$ operator we can identify such a form with a $1$-current $\star J u$. In our specific context, the $\star$ operator can be defined, at the level of vectors/covectors, as follows: given a $(n-1)$-covector $w$, the vector $\star w$ is defined by the identity
\[
v \cdot \star w = (v \wedge w) (e_1 \wedge \dots \wedge e_n) \quad \text{for all } v \in \wedge^1(\R^n).
\]

Jacobians turn out to be the main tool in our analysis due to their relation with boundaries. In order to highlight such a relation we need some additional notation: given any segment $S$ in $\R^n$ and given $\delta,\gamma > 0$, let us define the set
\[
U(S,\delta,\gamma) = \left\{ x \in \R^n \,:\, \text{dist}(x,S) < \min \left\{ \delta, \frac{\gamma}{\sqrt{1+\gamma^2}}\text{dist}(x,\partial S) \right\} \right\}.
\]
If we identify the line spanned by $S$ with $\R$, we can write each point $x \in U(S,\delta,\gamma)$ as $x = (x',x'') \in \R \times \R^{n-1}$, so that
$$
U(S,\delta,\gamma) = \{ x' \in S \,:\, |x''|\leq \min( \delta, \gamma\cdot\textup{dist}(x',\partial S)) \}.
$$
We can now recall the main result of \cite{ABO1} (rewritten in our specific context).

\begin{theorem}[Theorem 5.10, \cite{ABO1}]\label{teo:prescribedsing}
	Let $M = [\![\Sigma, \tau, 1]\!]$ be the (polyhedral) boundary of a polyhedral current $N$ of dimension $2$ in $\R^n$, and let $F_0$ denote the union of the faces of $N$ of dimension $0$. Then there exists $u \in W^{1,n-2}(\R^n;\mathbb S^{n-2})$ such that $\star J u = \alpha_{n-1}M$, with $u$ locally Lipschitz in the complement of $\Sigma \cup F_0$ and constant outside a bounded neighbourhood of $N$, and $Du$ belongs to $L^p$ for every $p < n-1$ and satisfies $|Du(x)| = O(1/\textup{dist}(x,\Sigma\cup F_0))$. Moreover, there exist $\delta,\gamma > 0$ small enough such that, for each $1$-simplex $S_k \subset \Sigma$, one has
	\[
	u(x) = \frac{x''}{|x''|} \quad \text{for all } x = (x',x'') \in U(S_k, \delta,\gamma).
	\]
\end{theorem}

%
%
%

\section{Gilbert--Steiner problems and currents\label{sec:recall}}

In this section we briefly review (this time in terms of currents) the approach used in \cite{BoOrOu, BoOu}, which is to say the framework introduced by Marchese and Massaccesi in \cite{MaMa,MaMa2}, and describe Gilbert--Steiner problems in terms of a minimum mass problem for a given family of rectifiable $1$-currents in $\R^n$.

The set of possible minimizers for $(I_\alpha)$ can be reduced to the set of (connected) acyclic graphs $L$ that are described as the superposition of $N-1$ curves.
\begin{definition}
	We define $\mathcal{G}(A)$ to be the set of acyclic graphs $L$ of the form
	\[
	L = \bigcup_{i=1}^{N-1} \lambda_i,
	\]
	where each $\lambda_i$ is a simple rectifiable curve connecting $P_i$ to $P_N$ and oriented by an $\mathcal{H}^1$-measurable unit vector field $\tau_i$, with $\tau_i(x) = \tau_j(x)$ for $\mathcal{H}^1$-a.e. $x \in \lambda_i \cap \lambda_j$, and we denote by $\tau$ the corresponding global orientation, i.e., $\tau(x) = \tau_i(x)$ for $\mathcal{H}^1$-a.e. $x \in \lambda_i$.
\end{definition}
\noindent
It can be shown (see, e.g., \cite[Lemma 2.1]{MaMa}), that $(I_\alpha)$ is equivalent to
\begin{equation}\label{eq:redIa}
\min \left\{ \int_L |\theta(x)|^\alpha d{\mathcal H}^1, \quad L \in \mathcal{G}(A),
\;\;\theta(x) = \sum_{i=1}^{N-1} \mathbf{1}_{\lambda_i}(x) \right\}.
\end{equation}

Given now $L \in \mathcal{G}(A)$, we identify each component $\lambda_i$ with the corresponding $1$-current $\Lambda_i = [\![ \lambda_i, \tau_i, 1 ]\!]$ and we consider $\Lambda = (\Lambda_1,\dots,\Lambda_{N-1}) \in [\mathcal I_1(\R^n)]^{N-1}$.
\begin{definition}
	We define $\mathcal{L}(A)$ to be the set $\Lambda \in [\mathcal I_1(\R^n)]^{N-1}$ such that each component is of the form $\Lambda_i = [\![ \lambda_i, \tau_i, 1 ]\!]$ for some $L \in \mathcal{G}(A)$, and write $\Lambda \equiv \Lambda_L$ to highlight the supporting graph.
\end{definition}
\noindent
Given $\Lambda = (\Lambda_1,\dots,\Lambda_{N-1}) \in [\mathbb N_1(\R^n)]^{N-1}$ and a function $\phi \in C^\infty_c(\R^d; \R^{d\times N-1})$, with $\phi = (\phi_1, \dots, \phi_{N-1})$, one sets
\[
\langle \Lambda, \phi\rangle = \sum_{i=1}^{N-1} \langle \Lambda_i, \phi_i \rangle
\]
and for a norm $\Psi$ on $\R^{N-1}$, we define the $\Psi$-mass measure of $\Lambda$ as
\begin{equation}\label{eq:psitv}
|\Lambda|_\Psi(\Omega):=\sup_{\substack{\omega \in C^\infty_c(\Omega;\R^n) \\ h \in C^{\infty}_c(\Omega;\R^{N-1})}} \left\{   \langle \Lambda, \omega\otimes h\rangle \, , \quad |\omega(x)|\le 1\, ,\  \Psi^*(h(x))\le 1 \right\}\, ,
\end{equation}
for $\Omega \subset \R^n$ open, where $\Psi^*(y) = \sup_{x \in \R^{N-1}} \langle y, x \rangle - \Psi(x)$ is the dual norm to $\Psi$ w.r.t. the scalar product on $\R^{N-1}$, and we let the $\Psi$-mass norm of $\Lambda$ to be
\begin{equation}\label{eq:psimass}
||\Lambda||_\Psi=|\Lambda|_\Psi(\R^n).
\end{equation}
As described in \cite{MaMa, BoOrOu, BoOu}, the problem defined in \eqref{eq:redIa} is equivalent to
\begin{equation}\label{eq:massmin}
\inf \{ ||\Lambda||_{\Psi_\alpha} \,:\, \Lambda = (\Lambda_1,\dots,\Lambda_{N-1}) \in [\mathcal I_1(\R^n)]^{N-1}, \;\partial \Lambda_i = \delta_{P_N} - \delta_{P_i}\},
\end{equation}
where $\Psi_\alpha$ is the $\ell^{1/\alpha}$ norm on $\R^{N-1}$ for $0 < \alpha \leq 1$, and the $\ell^\infty$ norm for $\alpha = 0$. This means that any minimizer $\bar\Lambda$ of \eqref{eq:massmin} is of the form $\bar\Lambda = \Lambda_{\bar L}$ for a minimizer $\bar L$ of \eqref{eq:redIa}, and given any minimizer $\bar L$ of \eqref{eq:redIa} then the corresponding $\Lambda_{\bar L}$ minimizes \eqref{eq:massmin}.

\begin{remark}
	In \cite{MaMa, MaMa2} problem \eqref{eq:massmin} is introduced in the context of a mass minimization problem for integral currents with coefficients in a suitable normed group. In that case, the $\Psi$-mass defined above is simply the mass of the current deriving from the particular choice of the norm for the coefficients group.
\end{remark}

\textbf{Calibrations}. One of the main advantages of formulation \eqref{eq:massmin} is the possibility to introduce calibration-type arguments for proving minimality of a given candidate. For a fixed $\bar\Lambda \in [\mathbb N_1(\R^n)]^{N-1}$, a (generalized) calibration associated to $\bar\Lambda$ is a linear and bounded functional $\phi \colon [\mathbb N_1(\R^n)]^{N-1} \to \R$ such that
\begin{itemize}
	\item[(i)] $\phi(\bar\Lambda) = ||\bar\Lambda||_\Psi$,
	\item[(ii)] $\phi(\partial R) = 0$ for any $R \in [\mathbb N_2(\R^n)]^{N-1}$,
	\item[(iii)] $\phi(\Lambda) \leq ||\Lambda||_\Psi$ for any $\Lambda \in [\mathbb N_1(\R^n)]^{N-1}$.
\end{itemize}
The existence of a calibration is a sufficient condition to prove minimality in \eqref{eq:massmin}. Indeed, let $\bar\Lambda$ be a competitor in \eqref{eq:massmin} and $\phi$ be a calibration for $\bar\Lambda$. Consider any $\Lambda \in [\mathbb N_1(\R^n)]^{N-1}$, with $\partial \Lambda_i = \delta_{P_N}-\delta_{P_i}$. By assumption, for each $i=1,\dots,N-1$, one has $\partial (\bar\Lambda_i-\Lambda_i) = 0$, so that there exists a $2$-current $R_i$ such that $\bar\Lambda_i = \Lambda_i + \partial R_i$. Hence,
\[
||\bar\Lambda||_\Psi \overset{\textup{(i)}}{=} \phi(\bar\Lambda) = \phi( \Lambda + \partial R) = \phi(\Lambda) + \phi(\partial R) \overset{\textup{(iii), (ii)}}{\leq} ||\Lambda||_\Psi
\]
which proves the minimality of $\bar\Lambda$ in \eqref{eq:massmin} (and, more generally, also minimality among normal currents). We also remark that once a calibration exists it must calibrate all minimizers.

\smallskip

\textbf{A calibration-type argument}. The general idea behind calibrations can be used to tackle minimality in suitable subclasses of currents, as long as the previous derivation can be proved to still hold true. Consider, as displayed in figure \ref{fig:graph3d4p}, the Steiner tree problem for four points in $\R^3$ with $P_1 = (-3/2, -\sqrt{3}/2, 0)$, $P_2 = (-3/2, \sqrt{3}/2, 0)$, $P_3 = (3/2, 0, \sqrt{3}/2)$ and $P_4 = (3/2, 0, -\sqrt{3}/2)$. Let us identify the two points $S_1 = (-1,0,0)$ and $S_2 = (1,0,0)$, and fix as norm $\Psi$ the $\ell^\infty$ norm on the coefficients space $\R^3$.
\begin{figure}
	\centering
	\includegraphics[width=0.5\linewidth]{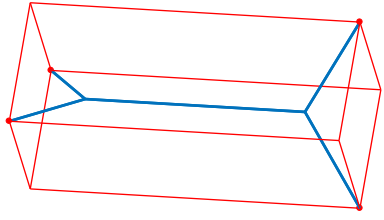}
	\caption{We consider the Steiner tree problem for $4$ vertices obtained as ``opposite'' couples of vertices of a rectangular cuboid.}
	\label{fig:graph3d4p}
\end{figure}
Given a list of points $Q_1,\dots,Q_k$, we write as $[Q_1, \dots, Q_k]$ the polyhedral current connecting them and oriented from $Q_1$ to $Q_k$. Our aim is to prove that
\[
\bar\Lambda = ([P_1, S_1, S_2, P_4], [P_2, S_1, S_2, P_4], [P_3, S_2, P_4])
\]
is a minimizer of the $\Psi$-mass $||\cdot||_\infty \equiv ||\cdot||_{\ell^\infty}$ among all currents $\Lambda \in \mathcal{B}$, where $\mathcal{B} \subset [\mathbb N_1(\R^3)]^3$ is the family of currents $\Lambda$ satisfying the given boundary conditions $\partial \Lambda_i = \delta_{P_4} - \delta_{P_i}$, and such that there exist a positive finite measure $\mu_\Lambda$ on $\R^3$, a unit vector field $\tau_\Lambda$ and a function $g^\Lambda \colon \R^3 \to \{e_1, e_2, e_3, e_1+e_2, e_1+e_2+e_3\}$ such that $\Lambda_i(\omega) = \int_{\R^3} g_i^\Lambda(x) \omega \cdot \tau_\Lambda \,d\mu_\Lambda$. Let us formally identify any such object as $\Lambda = (\tau_\Lambda \otimes g^\Lambda)\mu_\Lambda$ (loosely speaking, we consider only the family of normal rank one currents with a prescribed superposition pattern for different flows). It can be easily seen that $\bar\Lambda \in \mathcal{B}$ and for any $\Lambda \in \mathcal{B}$ we have $||\Lambda||_\infty = \int_{\R^3} ||g^\Lambda(x)||_\infty \,d\mu_\Lambda(x)$. For proving minimality of $\bar\Lambda$ for the $\ell^\infty$-mass among all competitors in $\mathcal{B}$ we can use a calibration argument: let us consider $\phi \colon [\mathbb N(\R^3)]^3 \to \R$ defined as
\[
\phi(\Lambda) = \sum_{i=1}^3 \langle \Lambda_i, \omega_i \rangle
\]
where $\omega_i$ are fixed to be
\[
\omega_1 = \frac12 dx^1 + \frac{\sqrt{3}}{2}dx^2, \quad \omega_2 = \frac12 dx^1 - \frac{\sqrt{3}}{2}dx^2, \quad \omega_3 = -\frac12 dx^1 - \frac{\sqrt{3}}{2}dx^3.
\]
One can show by direct computations that $\phi(\bar\Lambda) = ||\bar\Lambda||_\infty$, so that given any other $\Lambda \in \mathcal{B}$ and $R \in [\mathbb N_2(\R^3)]^3$ such that $\bar\Lambda = \Lambda + \partial R$, we have $||\bar\Lambda||_\infty = \phi(\bar\Lambda) = \phi(\Lambda) + \phi(\partial R)$, for which
\[
\phi(\Lambda) = \sum_{i=1}^3 \int_{\R^3} g_i^\Lambda(x) \omega_i \cdot \tau_\Lambda \,d\mu_\Lambda \leq \int_{\R^3} ||g^\Lambda||_\infty\,d\mu_\Lambda = ||\Lambda||_\infty
\]
because $g^\Lambda \in \{e_1,e_2,e_3,e_1+e_2,e_1+e_2+e_3\}$ for $\mu_\Lambda$-a.e. $x$, and
\[
\phi(\partial R) = \sum_{i=1}^3 \langle R_i, d\omega_i \rangle = 0.
\]
Hence, $||\bar\Lambda||_\infty \leq ||\Lambda||_\infty$ for any $\Lambda \in \mathcal{B}$. Up to permutations, the class $\mathcal{B}$ represents every possible acyclic graph $L \in \mathcal{G}(\{P_1,P_2,P_3,P_4\})$ with $2$ additional Steiner points and thus the support of $\bar\Lambda$ is an optimal Steiner tree within that family of graphs. Remark that any minimal configuration cannot have $0$ or $1$ Steiner points because these configurations violate the $120^\circ$ angle condition, so that we can conclude that the support of $\bar\Lambda$ is indeed an optimal Steiner tree. This extends for the first time to an higher dimensional context calibration-type arguments which up to now have been extensively used almost exclusively in the planar case, e.g. in \cite{MaMa, MaMa2}.


\medskip
In the companion paper \cite{BoOrOu}, we investigate a variational approximation of \eqref{eq:massmin} in the two dimensional case, relying on a further reformulation of the problem within a suitable family of $SBV$ functions and then providing a variational approximation based on Modica--Mortola type energies. Here, instead, we work in dimension three and higher and address \eqref{eq:massmin} directly by means of Ginzburg--Landau type energies.

%
%
%
%
%

\section{Variational approximation of $\Psi$-masses}\label{sec:varapp}

%
%

In this section we state and prove our main results, namely Proposition \ref{cor:minconpsi} and Theorem \ref{thm:main}, concerning the approximation of minimizers of $\Psi$-masses functionals through Jacobians of minimizers of Ginzburg--Landau type functionals, much in the spirit of \cite{ABO2}.

\subsection{Ginzburg--Landau functionals with prescribed boundary data}\label{subsec:singlejump}

In this section, following closely \cite{ABO2}, we consider Ginzburg--Landau functionals for functions having a prescribed trace $v$ on the boundary of a given open Lipschitz domain.

\smallskip
\textbf{Domain and boundary datum.} Fix two points $P, Q \in \R^n$, with $\max(|P|, |Q|) \leq 1$, and let $\Sigma$ be a simple acyclic polyhedral curve joining $P$ and $Q$, and oriented from $Q$ to $P$. Let $S_1,\dots,S_K$ be the $K$ segments composing $\Sigma$ and, for $\delta,\gamma>0$ small enough define
\begin{equation}\label{eq:lipdom}
U = \bigcup_{k=1}^K U(S_k, \delta,\gamma), \quad \text{and} \quad \Omega^{\delta,\gamma} = B_{10}^n \setminus \bar U.
\end{equation}
Consider the boundary datum $v \in W^{1-1/(n-1),n-1}(\partial\Omega^{\delta,\gamma};\mathbb S^{n-2})$ defined as
\begin{equation}\label{eq:bdyv}
v(x) =
\left\{
\begin{aligned}
&\frac{x''}{|x''|} &\quad& \text{for } x = (x',x'') \in \partial U \\
& e_{n-1} &\quad& \text{for }x \in \partial B_{10}^n \\
\end{aligned}
\right.
\end{equation}
By construction one has
\[
\star J v = \alpha_{n-1}(\delta_{Q}-\delta_{P}).
\]
In this context, for only two points, the $\Psi$-mass reduces (up to a constant) to the usual mass, and thus we can directly rely on Corollary 1.2 of \cite{ABO2}, which yields the following.

\begin{theorem}\label{thm:singlecomp}
	For $\delta,\gamma > 0$ small enough, consider the Lipschitz domain $\Omdg$ defined in \eqref{eq:lipdom} and let $v$ be the boundary datum defined in \eqref{eq:bdyv}.
	\begin{itemize}
		\item[(i)] Consider a (countable) sequence $\{u_{\eps}\}_\eps \subset W^{1,n-1}(\Omega^{\delta,\gamma};\R^{n-1})$ with trace $v$ on $\partial\Omega^{\delta,\gamma}$ such that $F_\eps(u_{\eps},\Omega^{\delta,\gamma}) = O(|\log\eps|)$. Then, up to subsequences, there exists a rectifiable $1$-current $M$ supported in $\bar{\Omega}^{\delta,\gamma}$, with $\partial M = \delta_{Q}-\delta_{P}$, such that the Jacobians $\star J u_\eps$ converge in the flat norm $\mathbf{F}_{\R^n}$ to $\alpha_{n-1}M$ and
		\begin{equation}\label{eq:liminfsingle}
		\liminf_{\eps\to 0} \frac{F_\eps(u_\eps, \Omega^{\delta,\gamma})}{|\log\eps|} \geq \beta_{n-1} ||M||
		\end{equation}
		\item[(ii)] Given a rectifiable $1$-current $M$ supported in $\bar{\Omega}^{\delta,\gamma}$ such that $\partial M = \delta_{Q}-\delta_{P}$, for every $\eps > 0$ we can find $u_\eps$ such that $u_\eps = v$ on $\partial{\Omega}^{\delta,\gamma}$, $\mathbf{F}_{\R^n}(\star J u_\eps - \alpha_{n-1}M) \to 0$ and
		\[
		\lim_{\eps\to 0} \frac{F_\eps(u_\eps, \Omega^{\delta,\gamma})}{|\log\eps|} = \beta_{n-1} ||M||
		\]
	\end{itemize}
	In particular, given $\{u_\eps\}_\eps$ a sequence of minimizers of $F_\eps(\cdot,\Omega^{\delta,\gamma})$ with trace $v$ on $\partial \Omega^{\delta,\gamma}$, then $F_\eps(u_\eps,\Omega^{\delta,\gamma}) = O(|\log \eps|)$ and, 
	possibly passing to a subsequence, the Jacobians $\star Ju_\eps$ converge in the flat norm $\mathbf F_{\R^n}$ to $\alpha_{n-1}M$, where $M$ minimizes the mass among all rectifiable $1$-currents supported on $\bar\Omega^{\delta,\gamma}$ with boundary $\delta_{Q}-\delta_{P}$.
\end{theorem}

Point $(i)$ of the previous theorem corresponds to the derivation of Section 3.1 in \cite{BoOrOu}, where we consider Modica--Mortola functionals for maps with prescribed jump, and here the prescribed jump is somehow replaced by the prescribed boundary datum ``around'' the drift $\Sigma$.
As it is done in \cite{BoOrOu}, the idea is now to extend the previous (single-component) result to problems involving $\Psi$-masses for $N \geq 3$.


\subsection{The approximating functionals $F_\eps^\Psi$}\label{subsec:muljump}

We now consider Ginzburg--Landau approximations for $\Psi$-masses whenever we are given $N \geq 3$ points. Fix then a norm $\Psi \colon \R^{N-1}\to [0,+\infty)$ on $\R^{N-1}$, and consider the $\Psi$-mass defined in \eqref{eq:psimass}.

\smallskip
\textbf{Construction of the domain.}
Fix a family of $N-1$ simple polyhedral curves $\gamma_i$ each one connecting $P_i$ to $P_N$ and denote by $\Gamma_i =[\![\gamma_i, \tau_i, 1]\!]$ the associated $1$-current (oriented from $P_N$ to $P_i$). Suppose, without loss of generality, that $\gamma_i \cap \gamma_j = \{P_N\}$ for any $i\neq j$, i.e., any two curves do not intersect each other. Every $\gamma_i$ can then be viewed as the concatenation of $m_i$ (oriented) segments $S_{i,1}, \dots, S_{i,m_i}$, for each of which we consider the neighbourhood
$$
U_{i,j}^{\delta,\gamma} = U(S_{i,j}, \delta, \gamma)
$$
for $\delta,\gamma > 0$.
Define now $V_i^{\delta,\gamma} = \cup_{j} U_{i,j}^{\delta,\sigma}$ and observe that, by finiteness, we can fix $\delta,\gamma$ sufficiently small such that $\bar V_i^{\delta,\gamma} \cap \bar V_j^{\delta,\gamma} = \{P_N\}$ for any $i \neq j$. The domain we are going to work with is
\begin{equation}\label{eq:domain}
\Omdg = B_{10}^n \setminus \left( \cup_i \bar V_i^\dg \right)
\end{equation}

\smallskip
\textbf{Boundary datum and approximating functionals.} Following the same idea used in the previous section, fix $N-1$ functions $v_i \in W^{1-1/(n-1),n-1}(\partial\Omega^\dg;\mathbb S^{n-2})$ such that
\[
v_i(x) =
\left\{
\begin{aligned}
& \frac{x''}{|x''|} &\quad& \text{for }x = (x',x'') \in \partial U_{i,j}^\dg \\
& e_{n-1} &\quad& \text{for }x \in \partial \Omega^{\delta,\gamma} \setminus \partial V_{i}^\dg \\
\end{aligned}
\right.
\]
By construction $v_i$ ``winds around'' $\gamma_i$ and is constant on the rest of the given boundary. As such, one sees that $\star J v_i = \alpha_{n-1}(\delta_{P_N}-\delta_{P_i})$. As our functional space we consider
\begin{equation}\label{eq:H}
H_i^\dg = \{u \in W^{1,n-1}(\Omega^\dg; \R^{n-1}) \,:\,u|_{\partial\Omega^\dg} = v_i\}, \quad H^\dg = H_1^\dg\times\dots\times H_{N-1}^\dg,
\end{equation}
and for $U = (u_1,\dots,u_{N-1}) \in H^\dg$ and $\vec e_\eps(U) = (e_\eps(u_1), \dots, e_\eps(u_{N-1}))$, we define the approximating functionals
\begin{equation}\label{eq:fpsieps}
{F}_\eps^\Psi(U,\Omega^\dg) = |\vec e_\eps (U) \,dx|_{\Psi}(\Omega^\dg),
\end{equation}
or equivalently, thanks to \eqref{eq:psitv},
\begin{equation}\label{eq:fpsiepsdef}
{ F}_\eps^\Psi(U,\Omega^\dg) = \sup_{\substack{\phi \in C^{\infty}_c(\Omega^\dg;\R^{N-1})}} \left\{ \sum_{i=1}^{N-1} \int_{\Omega^\dg} \phi_i e_\eps(u_i)\,dx, \quad \Psi^*(\phi(x))\le 1 \right\}.
\end{equation}

\smallskip
\textbf{Lower-bound inequality} Results on ``compactness'' and lower-bound inequality presented in the previous section extends to $F_\eps^\Psi$ as follows.

\begin{proposition}\label{prop:comp}
Consider a (countable) sequence $\{U_{\eps}\}_\eps \subset H^\dg$ such that $F^\Psi_\eps(U_{\eps},\Omega^{\delta,\gamma}) = O(|\log\eps|)$. Then, up to subsequences, there exists a family $M = (M_1,\dots,M_{N-1})$ of rectifiable $1$-currents supported in $\bar{\Omega}^{\delta,\gamma}$, with $\partial M_i = \delta_{P_N}-\delta_{P_i}$, such that the Jacobians $\star J u_{\eps,i}$ converge in the flat norm $\mathbf{F}_{\R^n}$ to $\alpha_{n-1}M_i$ and
\begin{equation}\label{eq:gammaliminfpsi}
\liminf_{\eps\to 0} \frac{F_\eps^\Psi(U_\eps, \Omega^{\delta,\gamma})}{|\log\eps|} \geq \beta_{n-1} ||M||_\Psi.
\end{equation}
\end{proposition}

\begin{proof}
For each $i = 1,\dots,N-1$, by definition of $F_\eps^\Psi$ we have
\[
\int_{\Omega^{\delta,\gamma}} e_\eps(u_{\eps,i})\,dx \leq \Psi^*(e_i) F_\eps^\Psi(U_\eps,\Omega^{\delta,\gamma}) = O(|\log \eps|)
\]
and the first part of the statement follows applying Proposition \ref{thm:singlecomp} componentwise. Fix now $\phi \in C^\infty_c(\R^n; \R^{N-1})$ with $\phi_i \geq 0$ for any $i = 1,\dots,N-1$ and $\Psi^*(\phi(x)) \leq 1$ for all $x$. Then, thanks to \eqref{eq:liminfsingle}, we have
\[
\begin{aligned}
\beta_{n-1} \sum_{i=1}^{N-1} \langle M_i, \phi_i \rangle  &\leq \frac{1}{|\log\eps|}\sum_{i=1}^{N-1} \liminf_{\eps\to 0} \int_{\Omega^{\delta,\gamma}} \phi_i e_\eps(u_{\eps,i}) \, dx \\
& \leq \frac{1}{|\log\eps|} \liminf_{\eps\to 0} \sum_{i=1}^{N-1} \int_{\Omega^{\delta,\gamma}} \phi_i e_\eps(u_{\eps,i}) \, dx \leq \liminf_{\eps\to 0} \frac{F_\eps^\Psi(U_\eps,\Omega^{\delta,\gamma})}{|\log\eps|},
\end{aligned}
\]
which yields \eqref{eq:gammaliminfpsi} taking the supremum over $\phi$.


\end{proof}

\smallskip
\textbf{Upper-bound inequality and behaviour of minimizers.} We now state and prove a version of an upper-bound inequality for the functionals $F_\eps^\Psi$ which is tailored to investigate the behaviour of Jacobians of minimizers of $F_\eps^\Psi$.

\begin{proposition}[Upper-bound inequality]\label{prop:limsup}
Let $\Lambda = \Lambda_L \in \mathcal{L}(A)$, with $L \in \mathcal{G}(A)$ an acyclic graph supported in $\Omdgb$. Then there exists a sequence $\{U_\eps\}_\eps \subset H^\dg$ such that $\mathbf F_{\R^n}(\star Ju_{\eps,i} - \alpha_{n-1}\Lambda_i) \to 0$, and
\begin{equation}\label{eq:gammalimsup}
\limsup_{\eps\to 0} \frac{F_\eps^\Psi(U_\eps, \Omdg)}{|\log \eps|} \leq \beta_{n-1} ||\Lambda||_\Psi.
\end{equation}
\end{proposition}


\begin{proof}
\emph{Step 1.} We assume that $L = \cup_i \lambda_i \in \mathcal{G}(A)$ is an acyclic polyhedral graph fully contained in $\Omega^{\delta,\gamma}$, which is to say $\lambda_i\cap \partial\Omega^{\delta,\gamma} = \{P_i,P_N\}$, and let $\tau$ be its global orientation. Such a graph $L$ can then be decomposed into a family of $K$ oriented segments $S_1, \dots, S_K$, with orientation given by $\tau$. For each segment $S_k$ consider the set $U_k^{'} = U(S_k, \delta', \gamma')$, for parameters $0 < \delta' < \delta$ and $0 < \gamma' < \gamma$, and choose $\delta',\gamma'$ small enough so that sets $U_k'$ are pairwise disjoint. Define as $V_i^{'}$ the union of the $U_k^{'}$ covering $\lambda_i$, and let $V' = \cup_i V_i' = \cup_k U_k'$. Eventually, define vectors $g^k \in \R^{N-1}$ as $g^k_i = 1$ if $S_k \subset \lambda_i$ and $g^k_i = 0$ otherwise. Collect these vectors in a function $g \colon V' \to \R^{N-1}$ defined as $g(x) = g^k$ for $x \in U_k'$.

For the construction of the approximating sequence we relay on the following fact, which is a direct consequence of Theorem \ref{teo:prescribedsing}: for each $i = 1,\dots,N-1$ there exists $u_i \in W^{1,n-2}(\Omdg;\mathbb S^{n-2})$ and a finite set of points $F_{0}^i$ such that:
	\begin{itemize}
		\item[(i)] $u_i|_{\partial \Omega^\dg} = v_i$, which is to say $u_i$ satisfies the given boundary conditions, and furthermore $\star Ju_i = \alpha_{n-1}\Lambda_i$; 
		\item[(ii)] $u_i$ is locally Lipschitz in $\Omdgb \setminus (\lambda_i \cup F_0^i)$ and
		\[
		|Du_i(x)| = O(1/\text{dist}(x,\gamma_i \cup \lambda_i \cup F_0^i));
		\]
		\item[(iii)] within the set $V'$ every function behaves like
		\[
		u_i(x) =
		\left\{
		\begin{aligned}
		& \frac{x''}{|x''|} &\quad& \text{for }x = (x',x'') \in V_{i}^{'} \\
		& e_{n-1} &\quad& \text{for }x \in \Omega^{\delta,\gamma} \setminus V_{i}^{'} \\
		\end{aligned}
		\right.
		\]
	\end{itemize}
	In particular, we observe that for any $k \in \{1,\dots, M\}$, if $S_k \subset \lambda_i$ and $S_k \subset \lambda_j$, then $u_i = u_j$ on $U_k^{'}$ by (iii). Thus, we can define a ``global'' function $u \colon V' \to S^{n-2}$ such that $u(x) = x''/|x''|$ for any $x \in V'$ and, consequently, $u_i|_{V'} = g_i(x)u(x)$.
	
	Starting form each $u_i$ we define our family of approximating maps: for any $\eps \in (0,\delta')$ let $\Omega_\eps^{\delta,\gamma} = \Omega^{\delta,\gamma} \setminus \cup_i B_{2\eps}(P_i)$, and let $u_{\eps,i} \colon \Omdg_\eps \to \R^{n-1}$ be defined as
	\begin{equation}\label{eq:defueps}
	u_{\eps,i}(x) = h_{\eps,i}(x)u_i(x) \quad \text{where} \quad h_{\eps,i}(x) = \min \left( 1, \frac{\text{dist}(x,\lambda_i \cup F_0^i)}{\eps} \right).
	\end{equation}
	Complete these maps on $B_{2\eps}(P_i) \cap \Omdg$ by means of a Lipschitz extension of the function $u_{\eps,i}$ with Lipschitz constant of the order of $1/\eps$,
	using $v_i$ as boundary value on $B_{2\eps}(P_i) \cap \partial\Omega^\dg$. 
	The resulting maps are locally Lipschitz in the complement of $\cup_k \partial S_k$, belong to $W^{1,n-1}(\Omdg; \R^{n-1})$ and by construction $u_{\eps,i}|_{\partial \Omdg} = v_i$, i.e., $u_{\eps,i} \in H_i^\dg$. Each $u_{\eps,i}$ converges strongly to $u_i$ in $W^{1,n-2}(\Omdg; \R^{n-1})$ and, in particular, the Jacobians $\star Ju_{\eps,i}$ converge to $\star Ju_i = \alpha_{n-1}\Lambda_i$ in the flat norm $\mathbf F_{\R^n}$ (see Remark 2.11 of \cite{ABO2}).
	
	We now consider the energy behaviour, working locally on every $U_k^{'}$: for $\eps \in (0,\delta')$, let us consider
	\[
	\begin{aligned}
	U_{k,\eps,1}^{'} &:= \{x \in U_k^{'} \,:\, \text{dist}(x, S_k) \leq \eps\} \cap \Omdg_\eps \\
	U_{k,\eps,2}^{'} &:= (U_k^{'} \setminus U_{k,\eps,1}^{'}) \cap \Omdg_\eps \\
	V_{out} &:= \Omdg_\eps \setminus V^{'} \\
	\end{aligned}
	\]
	Let $\phi = (\phi_1,\dots,\phi_{N-1})$, with $\phi_i\geq 0$ and $\Psi^*(\phi) \leq 1$, we compute
	\[
	\begin{aligned}
	\int_{\Omega^{\delta,\gamma}} \sum_{i=1}^{N-1} \phi_i e_{\eps}(u_{\eps,i}) \,dx &\leq \sum_{k=1}^K \left[ \int_{U_{k,\eps,1}'} \sum_{i=1}^{N-1} \phi_i e_{\eps}(u_{\eps,i}) \,dx + \int_{U_{k,\eps,2}'} \sum_{i=1}^{N-1} \phi_i e_{\eps}(u_{\eps,i}) \,dx \right] +  \\
	& + \sum_{j=1}^{N} \int_{B_{2\eps}(P_j)} \sum_{i=1}^{N-1} \phi_i e_{\eps}(u_{\eps,i}) \,dx + \int_{V_{out}} \sum_{i=1}^{N-1} \phi_i e_{\eps}(u_{\eps,i}) \,dx. \\
	\end{aligned}
	\]
	Fix $1 \leq k \leq K$ and consider the sets of indices $I_k = \{ i\,:\, S_k \subset \gamma_i \}$ and $I_k^c = \{1,\dots,N-1\} \setminus I_k$. Let us analyse separately the four kinds of integrals appearing in the above expression. 
	\begin{itemize}
		\item The first family of integrals on each $U_{k,\eps,1}'$ splits as
		\[
		\int_{U_{k,\eps,1}'} \sum_{i=1}^{N-1} \phi_i e_{\eps}(u_{\eps,i}) \,dx = \int_{U_{k,\eps,1}'} \sum_{i \in I_k} \phi_i e_{\eps}(u_{\eps,i}) \,dx + \int_{U_{k,\eps,1}'} \sum_{i \in I_k^c} \phi_i e_{\eps}(u_{\eps,i}) \,dx.
		\]
		We distinguish between two case.
		
		\emph{Case $i \in I_k$}: we have $|Du_i(x)| \leq C/\text{dist}(x,S_k)$ thanks to (iii), and therefore
		\[
		|Du_{\eps,i}(x)| \leq h_{\eps,i}(x)|Du_i(x)| + |Dh_{\eps,i}(x)||u_i(x)| \leq \frac{C}{\eps}.
		\]
		Using that $W(u_{\eps,i}) \leq C$ and $|U_{k,\eps,1}'| \leq C\eps^{n-1}$, we obtain
		\begin{equation}\label{eq:stima1}
		F_\eps(u_{\eps,i}, U_{k,\eps,1}') \leq C \quad \text{for all } k, i \text{ such that } S_k \subset \lambda_i.
		\end{equation}
		
		\emph{Case $i \in I_k^c$}: in this situation we have $u_{\eps,i} = u_i$ on $U_{k,\eps,1}'$ and $\text{dist}(x,F^i_0) \leq C\text{dist}(x,\gamma_i \cup \lambda_i)$. In particular, combining (ii) and \eqref{eq:defueps}, we have
		$$
		|Du_{\eps,i}(x)| \leq C/\text{dist}(x,F^i_0).
		$$
		Using the fact that $W(u_{\eps,i}) = 0$ in the complement of an $\eps$-neighbourhood $(\lambda_i \cup F^i_0)_\eps$ of $\lambda_i \cup F^i_0$, we get
		\begin{equation}\label{eq:stima2}
		\begin{aligned}
		F_\eps(u_{\eps,i}, U_{k,\eps,1}') &\leq C\int_{U_{k,\eps,1}'} \frac{dx}{\text{dist}(x,F^i_0)^{n-1}} + \frac{C}{\eps^2}|(\lambda_i \cup F^i_0)_\eps| \\
		&\leq C  \quad \text{for all } k, i \text{ such that } S_k \nsubseteq \lambda_i.
		\end{aligned}
		\end{equation}
		Combining \eqref{eq:stima1} and \eqref{eq:stima2} we obtain
		\begin{equation}\label{eq:stimaint1}
		\int_{U_{k,\eps,1}'} \sum_{i=1}^{N-1} \phi_i e_{\eps}(u_{\eps,i}) \,dx \leq C \quad \text{for all } 1 \leq k \leq K, 1 \leq i \leq N-1.
		\end{equation}
		
		\item The second family of integrals on each $U_{k,\eps,2}'$ splits analogously into
		\[
		\int_{U_{k,\eps,2}'} \sum_{i=1}^{N-1} \phi_i e_{\eps}(u_{\eps,i}) \,dx = \int_{U_{k,\eps,2}'} \sum_{i \in I_k} \phi_i e_{\eps}(u_{\eps,i}) \,dx + \int_{U_{k,\eps,2}'} \sum_{i \in I_k^c} \phi_i e_{\eps}(u_{\eps,i}) \,dx.
		\]
		Let us distinguish the same two cases as above.
		
		\emph{Case $i \in I_k$}: here we have $u_{\eps,i} = u_i$ within $U_{k,\eps,2}'$ and so $u_{\eps,i}$ takes values in $S^{n-2}$, reducing this way $e_\eps(u_{\eps,i})$ to $\frac{1}{n-1}|Du_i|^{n-1}$. For every $x \in U_k'$ one has
		\[
		|Du_i(x)| = \left|D \frac{x''}{|x''|}\right| = \frac{(n-2)^{1/2}}{|x''|}.
		\]
		Hence,
		\begin{equation}\label{eq:stima3}
		\begin{aligned}
		F_\eps(u_{\eps,i}, U_{k,\eps,2}') &\leq \mathcal{H}^1(S_k) \frac{(n-2)^{(n-1)/2}}{n-1} \int_{B_{\delta'}^{n-1} \setminus B_\eps^{n-1}} \frac{dx''}{|x''|^{n-1}} \\
		& \leq \mathcal{H}^1(S_k) \frac{(n-2)^{(n-1)/2}}{n-1} \int_\eps^1 \frac{(n-1)\alpha_{n-1}\rho^{n-2}}{\rho^{n-1}} \,d\rho \\
		&\leq \beta_{n-1} |\log \eps|\cdot \mathcal{H}^1(S_k)  \quad \text{for all } k, i \text{ such that } S_k \subset \lambda_i.
		\end{aligned}
		\end{equation}
		
		\emph{Case $i \in I_k^c$}: the same derivation done for obtaining \eqref{eq:stima2} applies, so that
		\begin{equation}\label{eq:stima4}
		F_\eps(u_{\eps,i}, U_{k,\eps,2}') \leq C \quad \text{for all } k, i \text{ such that } S_k \nsubseteq \lambda_i.
		\end{equation}
		
		Taking into account \eqref{eq:stima3}, \eqref{eq:stima4}, and that $\sum_{i\in I_k} \phi_i(x) = \sum_{i=1}^{N-1} g_i^k\phi_i(x) \leq \Psi(g^k)$, we have
		\begin{equation}\label{eq:stimaint2}
		\int_{U_{k,\eps,2}'} \sum_{i=1}^{N-1} \phi_i e_{\eps}(u_{\eps,i}) \,dx \leq C + \Psi(g^k)\beta_{n-1} |\log \eps|\cdot \mathcal{H}^1(S_k)
		\end{equation}
		for all $1 \leq k \leq K$, $1 \leq i \leq N-1$.
		
		\item For any given $j = 1,\dots,N$ the contribution on $B_{2\eps}(P_j)$ is of order $\eps$, so that in particular
		\begin{equation}\label{eq:stimaint3}
		\int_{B_{2\eps}(P_j)} \sum_{i=1}^{N-1} \phi_i e_{\eps}(u_{\eps,i}) \,dx \leq C.
		\end{equation}
		
		\item The last integral on $V_{out}$ can be treated as in the derivation of \eqref{eq:stima2} and \eqref{eq:stima4}, so that we have
		\begin{equation}\label{eq:stimaint4}
		\int_{V_{out}} \sum_{i=1}^{N-1} \phi_i e_{\eps}(u_{\eps,i}) \,dx \leq C \quad \text{for all } 1 \leq k \leq K, 1 \leq i \leq N-1.
		\end{equation}

	\end{itemize}
If we combine \eqref{eq:stimaint1}, \eqref{eq:stimaint2}, \eqref{eq:stimaint3}, \eqref{eq:stimaint4}, divide by $|\log\eps|$, take $\eps \to 0$ and consider the supremum over $\phi$ in view of \eqref{eq:fpsiepsdef}, we have
\[
\limsup_{\eps\to 0} \frac{F^\Psi_\eps(U_\eps,\Omega^\dg)}{|\log\eps|} \leq \beta_{n-1} |\Lambda|_\Psi(\Omega^\dg) = \beta_{n-1} ||\Lambda||_\Psi,
\]
which is the sought for conclusion.

\medskip

\emph{Step 2.} Let us consider now the case $\Lambda_L \equiv \Lambda = (\Lambda_1,\dots,\Lambda_{N-1})$, $L = \cup_i \lambda_i$ and the $\lambda_i$ are not necessarily polyhedral and possibly lying on the boundary of $\Omdg$. We rely on Lemma \ref{lemma:polyapprox} below to construct a sequence of acyclic polyhedral graphs $L_m = \cup_i \lambda_{i}^m$, $\lambda_i^m$ contained in $\Omdg$, and s.t. the Hausdorff distance $d_H(\lambda_i^m,\lambda_i) < \frac1m$ for all $i = 1,\dots,N-1$, and $||\Lambda_{L_m}||_\Psi \leq ||\Lambda_L||_\Psi + \frac1m$. For $\Lambda_{L_m} = (\Lambda_1^m, \dots, \Lambda_{N-1}^m)$, by step 1 we may construct a sequences $\{U^m_\eps\}_\eps$ such that $\mathbf{F}_{\R^n}(\star Ju_{\eps,i}^m - \alpha_{n-1}\Lambda_i^m) \to 0$ as $\eps \to 0$ for each $m$ and, in particular, 
\[
\begin{aligned}
\limsup_{\eps\to 0} \frac{F_\eps^\Psi(U_\eps^m, \Omdg)}{|\log \eps|} \leq \beta_{n-1} ||\Lambda_{L_m}||_\Psi \leq \beta_{n-1} ||\Lambda||_\Psi + \frac{C}{m}.
\end{aligned}
\]
We deduce that $\mathbf{F}_{\R^n}(\star Ju_{\eps_m,i}^m - \alpha_{n-1}\Lambda_i) \to 0$ and
\[
\limsup_{m\to \infty}  \frac{F_{\eps_m}^\Psi(U_{\eps_m}^m, \Omdg)}{|\log\eps_m|} \leq \beta_{n-1} ||\Lambda||_\Psi
\]
for a subsequence $\eps_m \to 0$ as $m \to +\infty$. Conclusion \eqref{eq:gammalimsup} follows.
\end{proof}

We recall from \cite[Lemma 3.10]{BoOrOu} the relevant approximation used above, where polyhedral approximations are here supposed to live within the set $\Omdg$ (i.e., with no relevant part on the boundary).
\begin{lemma}\label{lemma:polyapprox}
Let $L \in \mathcal{G}(A)$,  $L = \cup_{i=1}^{N-1} \lambda_i$, be an acyclic graph connecting $P_1,\dots,P_N$ with $\lambda_i \subset \Omdgb$. Then for any $\eta > 0$ there exists $L' \in \mathcal{G}(A)$, $L' = \cup_{i=1}^{N-1} \lambda_i'$, with $\lambda'_i \subset \Omdg \cup \{P_i,P_N\}$ a simple polyhedral curve of finite length connecting $P_i$ to $P_N$, such that the Hausdorff distance $d_H(\lambda_i,\lambda_i') < \eta$ and $||\Lambda_{L'}||_\Psi \leq ||\Lambda_L||_\Psi + \eta$.
\end{lemma}

Thanks to the previous propositions we are now able to prove our main result on the behaviour of the Jacobians of the minimizers.

\begin{proposition}[Behaviour of minimizers]\label{cor:minconpsi} Let $\{U_\eps\}_\eps\subset H^{\delta,\gamma}$ be a sequence of minimizers for $F_\eps^\Psi$ in $H^\dg$. Then (up to a subsequence) the Jacobians $\star J u_{\eps,i}$ converge in the flat norm $\mathbf F_{\R^n}$ to $\alpha_{n-1}M_i$, with $M = (M_1,\dots,M_{N-1})$ a minimizer of
\begin{equation}\label{eq:mindg}
\inf \{ ||\Lambda||_{\Psi} \,:\, \Lambda = (\Lambda_1,\dots,\Lambda_{N-1}) \in [\mathcal{I}_1(\R^n)]^{N-1}, \textup{spt}\,\Lambda_i \subset \Omdgb, \partial \Lambda_i  = \delta_{P_N}-\delta_{P_i} \}.
\end{equation}
\end{proposition}

\begin{proof}
Let $\Lambda = \Lambda_L$ canonically representing an acyclic graph $L \subset \Omdgb$, and let $\{V_\eps\}_\eps \subset H^\dg$ such that $\limsup_{\eps\to 0} \frac{F_\eps^\Psi(V_\eps,\Omega^\dg )}{|\log\eps|}\le ||\Lambda||_\Psi$ and $\mathbf F_{\R^n}(\star Jv_{\eps,i} - \alpha_{n-1}\Lambda_i) \to 0$. Since $F_\eps^\Psi(U_\eps,\Omega^\dg )\le F_\eps^\Psi(V_\eps,\Omega^\dg )$, by Proposition \ref{prop:comp} there exists a family $M = (M_1,\dots,M_{N-1})$ of rectifiable $1$-currents supported in $\bar{\Omega}^{\delta,\gamma}$, with $\partial M_i = \delta_{P_N}-\delta_{P_i}$, such that the Jacobians $\star J u_{\eps,i}$ converge in the flat norm $\mathbf F_{\R^n}$ to $\alpha_{n-1}M_i$. Then, by \eqref{eq:gammaliminfpsi}, we have
$$
\beta_{n-1}||M||_\Psi\le \liminf_{\eps\to 0} \frac{F_\eps^\Psi(U_\eps,\Omega^\dg )}{|\log\eps|}\le\limsup_{\eps\to 0} \frac{F_\eps^\Psi(V_\eps,\Omega^\dg )}{|\log\eps|}\le \beta_{n-1}||\Lambda||_\Psi.
$$
Given any other generic $\Lambda \in [\mathcal{I}_1(\R^n)]^{N-1}$ with $\textup{spt}\,\Lambda_i \subset \Omdgb$ and $\partial \Lambda_i  = \delta_{P_N}-\delta_{P_i}$, as one does in the derivation of \eqref{eq:redIa} (see, e.g., Lemma 2.1 in \cite{MaMa}), we can always find $\bar{L} \in \mathcal{G}(A)$ supported in $\Omdgb$ such that $||\Lambda_{\bar L}||_\Psi \leq ||\Lambda||_\Psi$, and thus $M$ minimizes \eqref{eq:mindg} as desired.
\end{proof}

Finally, let us highlight the case $\Psi = \Psi_\alpha$, where $\Psi_\alpha(g) = |g|_{1/\alpha}$ for $0<\alpha\leq 1$ and $\Psi_0(g) = |g|_\infty$, and denote $F^0_\eps \equiv F^{\Psi_0}_\eps$ and $F^\alpha_\eps \equiv F^{\Psi_\alpha}_\eps$. For $U=(u_1,\dots,u_{N-1}) \in H^\dg$ we have
\begin{equation}\label{eq:Fepsalpha}
{ F}_\eps^0(U,\Omega^\dg)=\int_{\Omega^\dg} \, \sup_{i} e_\eps(u_i) \, dx, \quad\quad F_\eps^\alpha(U,\Omega^\dg)=\int_{\Omega^\dg} \, \left(\sum_{i=1}^{N-1} e_\eps(u_i)^{1/\alpha}\right)^\alpha\, dx.
\end{equation}

\begin{theorem}\label{thm:main}
Let $\{P_1,\dots,P_N\} \subset \R^n$ such that $\max_i |P_i|=1$, and let $\Omega^\dg$ be defined as in \eqref{eq:domain} for $\delta,\gamma$ small enough, with $\gamma = \bar{c}\delta$. For $0 \leq \alpha \leq 1$ and $0 < \eps \ll \delta$, denote $F^{\alpha,\delta}_\eps \equiv { F}_\eps^\alpha(\cdot,\Omega^\dg)$, with $F^\alpha_\eps(\cdot,\Omega^\dg)$ defined in \eqref{eq:Fepsalpha}.
\begin{itemize}
\item[(i)] Let $\{ U_\eps^{\alpha,\delta} \}_\eps$ be a sequence of minimizers for $F_\eps^{\alpha,\delta}$ in $H^\dg$, with $H^\dg$ defined in \eqref{eq:H}. Then, up to subsequences, the Jacobians $\star Ju_{\eps,i}^{\alpha,\delta}$ converge in the flat norm $\mathbf F_{\R^n}$ to $\alpha_{n-1}M_i^{\alpha,\delta}$, where $M^{\alpha,\delta} = (M_1^{\alpha,\delta}, \dots, M_{N-1}^{\alpha,\delta})$ minimizes \eqref{eq:mindg}.

\item[(ii)]Let $M^{\alpha,\delta} = (M_1^{\alpha,\delta}, \dots, M_{N-1}^{\alpha,\delta})$ be a sequence of minimizers for \eqref{eq:mindg}. Then, up to subsequences, we have $\mathbf F_{\R^n}(M_i^{\alpha,\delta} - M_i^\alpha) \to 0$ as $\delta \to 0$ for every $i = 1,\dots,N-1$, with $M^\alpha = (M_1^\alpha,\dots,M_{N-1}^\alpha)$ a minimizer of
\begin{equation}\label{eq:massmin2}
\inf \{ ||\Lambda||_{\Psi_\alpha} \,:\, \Lambda = (\Lambda_1,\dots,\Lambda_{N-1}) \in [\mathcal I_1(\R^n)]^{N-1}, \;\partial \Lambda_i = \delta_{P_N}-\delta_{P_i}\}
\end{equation}
and, in turn, $M^\alpha = \Lambda_{L_\alpha}$ for an optimizer $L_\alpha$ of the $\alpha$-irrigation problem $(I_\alpha)$ with terminals $P_1,\dots,P_N$.
\end{itemize}
\end{theorem}

\begin{proof}
In view of Proposition \ref{cor:minconpsi} it remains to prove item $(ii)$. For each $i = 1,\dots,N-1$, the sequence $\{ M^{\alpha,\delta}_i \}_\delta$ is equibounded in mass, hence there exists a rectifiable $1$-current $M_i^\alpha$, with $\partial M_i^\alpha = \delta_{P_N} - \delta_{P_i}$, such that $M_i^{\alpha,\delta} \to M_i^\alpha$ in the flat norm. Let us call $M^\alpha = (M_1^\alpha,\dots,M_{N-1}^\alpha)$ the limiting family and let $\bar M^\alpha = (\bar M_1^\alpha,\dots,\bar M_{N-1}^\alpha)$ be a minimizer of \eqref{eq:massmin2}. In the same spirit of Lemma \ref{lemma:polyapprox}, starting with our minimizer $\bar{M}^\alpha$, we can construct a new family $\tilde M^{\alpha,\delta} = (\tilde M_1^{\alpha,\delta},\dots,\tilde M_{N-1}^{\alpha,\delta})$ supported in $\Omdgb$ such that $||\tilde M^{\alpha,\delta}||_{\Psi^\alpha} \leq ||\bar M^{\alpha,\delta}||_{\Psi^\alpha} + C\delta$. Hence,
\[
\begin{aligned}
||\bar M^\alpha||_{\Psi^\alpha} &\leq ||M^\alpha||_{\Psi^\alpha} \leq \liminf_{\delta\to 0} ||M^{\alpha,\delta}||_{\Psi^\alpha} \leq \liminf_{\delta\to 0} ||\tilde M^{\alpha,\delta}||_{\Psi^\alpha} \\
& \leq \liminf_{\delta\to 0} ||\bar M^{\alpha,\delta}||_{\Psi^\alpha} + C\delta = ||\bar M^\alpha||_{\Psi^\alpha},
\end{aligned}
\]
and so $M^\alpha$ has to be a minimizer of \eqref{eq:massmin2}. The correspondence of minimizers of \eqref{eq:massmin2}, which is to say of \eqref{eq:massmin}, with minimizers of $(I_\alpha)$ follows by the discussion of Section \ref{sec:recall}.
\end{proof}

\section*{Acknowledgements}

The first and second author are partially supported by GNAMPA-INdAM. 
The third author gratefully acknowledges the support of the ANR through the
project GEOMETRYA, the project COMEDIC and  the LabEx PERSYVAL-Lab (ANR-11-LABX-0025-01).

\bibliographystyle{plain}
\bibliography{references}

\begin{thebibliography}{10}

\bibitem{ABO1}
Giovanni Alberti, Sisto Baldo, and Giandomenico Orlandi.
\newblock Functions with prescribed singularities.
\newblock {\em Journal of the European Mathematical Society}, 5(3):275--311,
  2003.

\bibitem{ABO2}
Giovanni Alberti, Sisto Baldo, and Giandomenico Orlandi.
\newblock Variational convergence for functionals of {G}inzburg-{L}andau type.
\newblock {\em Indiana Univ. Math. J.}, 54(5):1411--1472, 2005.

\bibitem{BeCaMo}
Marc Bernot, Vicent Caselles, and Jean-Michel Morel.
\newblock {\em Optimal transportation networks: models and theory}, volume
  1955.
\newblock Springer Science \& Business Media, 2009.

\bibitem{BoOrOu}
Mauro Bonafini, Giandomenico Orlandi, and \'{E}douard Oudet.
\newblock Variational approximation of functionals defined on $1$-dimensional
  connected sets: the planar case.
\newblock {\em SIAM J. Math. Anal.}, 50(6):6307--6332, 2018.

\bibitem{BoOu}
Mauro Bonafini and \'Edouard Oudet.
\newblock A convex approach to the gilbert--steiner problem.
\newblock 2018.

\bibitem{BoBrLe}
Matthieu Bonnivard, Elie Bretin, and Antoine Lemenant.
\newblock Numerical approximation of the steiner problem in dimension 2 and 3.
\newblock 2018.

\bibitem{Mi-al}
Matthieu Bonnivard, Antoine Lemenant, and Vincent Millot.
\newblock On a phase field approximation of the planar {S}teiner problem:
  existence, regularity, and asymptotic of minimizers.
\newblock {\em Interfaces Free Bound.}, 20(1):69--106, 2018.

\bibitem{BoLeSa}
Matthieu Bonnivard, Antoine Lemenant, and Filippo Santambrogio.
\newblock Approximation of length minimization problems among compact connected
  sets.
\newblock {\em SIAM J. Math. Anal.}, 47(2):1489--1529, 2015.

\bibitem{ChMeFe}
Antonin Chambolle, Luca Alberto~Davide Ferrari, and Benoit Merlet.
\newblock A phase-field approximation of the steiner problem in dimension two.
\newblock {\em Advances in Calculus of Variations}, 2017.

\bibitem{ChMeFe2}
Antonin Chambolle, Luca Alberto~Davide Ferrari, and Benoit Merlet.
\newblock Variational approximation of size-mass energies for k-dimensional
  currents.
\newblock {\em arXiv preprint arXiv:1710.08808}, 2017.

\bibitem{JeSo02}
Robert~L Jerrard and Halil~Mete Soner.
\newblock Functions of bounded higher variation.
\newblock {\em Indiana University mathematics journal}, pages 645--677, 2002.

\bibitem{MaMa2}
Andrea Marchese and Annalisa Massaccesi.
\newblock An optimal irrigation network with infinitely many branching points.
\newblock {\em ESAIM Control Optim. Calc. Var.}, 22(2):543--561, 2016.

\bibitem{MaMa}
Andrea Marchese and Annalisa Massaccesi.
\newblock The {S}teiner tree problem revisited through rectifiable
  {$G$}-currents.
\newblock {\em Adv. Calc. Var.}, 9(1):19--39, 2016.

\bibitem{MaOuVe}
Annalisa Massaccesi, \'Edouard Oudet, and Bozhidar Velichkov.
\newblock Numerical calibration of {S}teiner trees.
\newblock {\em Applied Mathematics \& Optimization}, pages 1--18, 2017.

\bibitem{OuSa}
Edouard Oudet and Filippo Santambrogio.
\newblock A {M}odica-{M}ortola approximation for branched transport and
  applications.
\newblock {\em Arch. Ration. Mech. Anal.}, 201(1):115--142, 2011.

\bibitem{Sandier}
Etienne Sandier.
\newblock Ginzburg-{L}andau minimizers from {$\Bbb R^{n+1}$} to {$\Bbb R^n$}
  and minimal connections.
\newblock {\em Indiana Univ. Math. J.}, 50(4):1807--1844, 2001.

\bibitem{Xia}
Qinglan Xia.
\newblock Optimal paths related to transport problems.
\newblock {\em Commun. Contemp. Math.}, 5(2):251--279, 2003.

\end{thebibliography}

\end{document}